\documentclass[12pt,a4paper]{amsart}

 \usepackage{amssymb,amsmath,amsfonts,amssymb, amsthm}
 \usepackage{graphics,graphicx,color}
 -\marginparwidth \oddsidemargin -4mm \evensidemargin
 \oddsidemargin%
 \advance\hoffset by 5mm
 \usepackage{amsmath}    
 \advance\hoffset by 5mm
 \numberwithin{equation}{section}

\newcommand\eps{{\varepsilon}}

\newcommand{\Rr}{{\mathbb R}}

\newcommand{\vfi}{\varphi}
\newcommand{\tvfi}{\tilde\varphi}
\newcommand{\be}{\begin{equation}}
\newcommand{\ee}{\end{equation}}

 \title[Weak formulations of the water waves equations]{Equivalence of weak formulations of the steady water waves equations}
 \date{\today}
\author{Eugen Varvaruca}
\address{Department of Mathematics and Statistics, University of Reading, Whiteknights Campus, PO Box 220, Reading RG6 6AX, United Kingdom}
\email{e.varvaruca@reading.ac.uk}
 \author{Arghir Zarnescu}
 \address{Mathematical Institute, University of Oxford, 24--29 St Giles', Oxford OX1 3LB, United Kingdom}
 \email{zarnescu@maths.ox.ac.uk}

 \newtheorem{lemma} {Lemma}
 
 \newtheorem{theorem} {Theorem}

\begin{document}
 \begin{abstract} We prove the equivalence of three weak formulations of the steady water waves equations, namely the velocity formulation, the stream function formulation, and the Dubreil-Jacotin formulation, under weak H\"{older} regularity assumptions on their solutions.
 \end{abstract}

 \maketitle
 \section{Introduction}

  We consider the classical problem of travelling waves at the free surface of a two-dimensional inviscid, incompressible, heavy fluid  over a flat bed. Over the past few years this problem has attracted considerable interest, see \cite{S} for a survey of recent developments.
In the most direct mathematical description, the problem is to find steady solutions of the incompressible Euler system for the velocity field $(u,v)$ and the pressure field $P$, together with relevant boundary conditions, in an unknown domain in the plane. In this form the problem is difficult to treat mathematically, and other, more convenient reformulations, have been used in the literature. One such reformulation involves a stream function $\psi$, whose existence is ensured by the incompressibility condition. Namely, the stream function $\psi$ satisfies a semilinear elliptic equation (since vorticity may be present in the flow), together with suitable boundary conditions. This reformulation still retains the difficulty of being a free-boundary problem. The most general approach to overcoming this difficulty  uses a change of variables introduced by Dubreil-Jacotin in 1934, which transforms the problem to an equivalent one in a strip. The Dubreil-Jacotin formulation is a cornerstone of the large and growing literature on large-amplitude travelling water waves with vorticity that began with the work of Constantin and Strauss \cite{csold}.
(For other methods, of more restricted applicability, of transforming the free-boundary problem for the stream function into a problem over a fixed domain, see
 \cite{CV}, \cite{EEW}, \cite{ST}.)

As we discuss in Section 2, following essentially \cite{csold}, the three formulations of the steady water waves equations just described, to which we refer to as the \emph{velocity formulation}, the \emph{stream function formulation} and the \emph{Dubreil-Jacotin (or height) formulation}, are equivalent when considered in the classical sense. However, the recent study \cite{cs} by Constantin and Strauss  has raised the question whether the three formulations are still equivalent when considered in a suitable weak sense which we describe in Section 3. The main result of \cite{cs} is a global bifurcation theory for solutions, periodic in the horizontal direction, of class $C^{1,\alpha}$ for some $\alpha\in (0,1)$ of the weak Dubreil-Jacotin equation. Such solutions would formally correspond to solutions $(u,v)$ and $P$, also periodic in the horizontal direction, and  of class $C^{0,\alpha}$, of the weak velocity formulation, in a domain with $C^{1,\alpha}$ boundary. However, the key question of the equivalence of the two weak formulations under these regularity assumptions is left as an open problem by \cite{cs}. (A result on the equivalence of the formulations is given in \cite{cs}, but under different, and not the most natural, regularity assumptions.)

 The  main result of the present paper, Theorem 1, which is given in Section 4, provides an affirmative answer to the above open problem, albeit only in the case when the H\"{o}lder exponent satisfies $\alpha >1/3$. More precisely, Theorem 1 proves that, under the above regularity assumptions, the weak velocity and the weak stream function formulations are equivalent for any $\alpha\in (1/3,1]$, while the weak stream function and the weak height formulations are equivalent for any $\alpha\in (0,1]$. An important consequence of our result is that, at least in the case $\alpha>1/3$, the solutions constructed in \cite{cs} of the Dubreil-Jacotin formulation are relevant (in the sense that they give rise to corresponding solutions) for the velocity formulation of the steady water waves equations.
Our result is in the same spirit as, and its proof is inspired by, the Onsager conjecture as proved (partially) in \cite{CET}. The Onsager conjecture is, essentially, the statement that solutions of the time-dependent incompressible Euler equations on a fixed domain (in dimension three, with no external forces), of class $C^{0,\alpha}$ in the space variables for each value of the time variable, conserve their energy in time if $\alpha>1/3$ and may fail to do so if $\alpha\leq 1/3$. The paper \cite{CET} proves that $\alpha>1/3$ implies conservation of energy (and leaves open the reverse statement in the conjecture). As in \cite{CET}, our proof is based on regularizing the equations and, roughly speaking, the assumption $\alpha>1/3$ is used in an essential way to show that certain remainder terms converge to $0$ as the regularization parameter tends to $0$.
An important problem left open by the present paper is that of whether the weak velocity formulation and the weak stream function formulation are also equivalent in the case when the H\"{o}lder exponent satisfies $\alpha\leq 1/3$.

 \section{Classical formulations of the steady water waves problem}
\subsection{The velocity formulation}
  We consider a wave travelling with constant speed and without change of shape on the free surface of a two-dimensional inviscid, incompressible fluid of unit density, acted on by gravity, over a flat, horizontal, impermeable bed. This means that, in a frame of reference moving at the speed $c$ of the wave, the fluid is in steady flow in a fixed domain. Let the free surface be given by $y=\eta(x)$, for some function $\eta:\Rr\to\Rr$, and the flat bottom be given by $y=0$, so that the fluid domain is
  $$D_\eta\stackrel{\rm def}{=}\{(x,y):x\in\mathbb{R},0<y<\eta(x)\}.$$
  Then the fluid motion is described, see \cite{csold} for details, by the following equations and boundary conditions for a steady velocity field $(u,v)$ and a pressure field $P$ in $D_\eta$:
 \begin{subequations}\label{Euler}
  \begin{align}
 u_x+v_y=0 \,\,\,\,\quad\qquad\qquad\qquad&\text{in }D_\eta,\label{eq:masscons}\\
 (u-c)u_x+vu_y=-P_x\,\qquad\qquad\qquad&\text{in }D_\eta,\label{eq:momcons1}\\
  (u-c)v_x+vv_y=-P_y-g\,\,\qquad\qquad&\text{in }D_\eta, \label{eq:momcons2}\\
  v=0 \,\,\,\quad\qquad\qquad\qquad&\text{on } y=0,\label{vb}\\
 v=(u-c)\eta_x  \qquad\qquad&\text{on } y=\eta(x),  \label{eq:velbdry}\\
  P=P_{\textnormal{atm}} \qquad\qquad\qquad&\text{on }y=\eta(x),\label{pres}
 \end{align}
  \end{subequations}
where $P_{\textnormal{atm}}$ is the constant atmospheric pressure and $g$ is the gravitational constant of acceleration. More precisely, (\ref{eq:masscons}) describes mass conservation, (\ref{eq:momcons1})--(\ref{eq:momcons2}) describe momentum conservation, (\ref{vb}) expresses the fact that the velocity at the bottom is horizontal, (\ref{eq:velbdry}) is the kinematic condition that the same particles always form the free surface, while (\ref{pres}) is the dynamic condition that
at the free surface the pressure in the fluid equals the constant atmospheric pressure. This is a free-boundary problem, because the domain $D_\eta$ is not known a priori. The system (\ref{Euler}) will be referred to as the \emph{velocity formulation} of the steady water waves equations. Throughout the paper we assume that, in the moving frame, the horizontal velocity of all the particles is in the same direction. For definiteness, we assume that
\be u<c\quad\text{in }\overline{D_\eta}.\label{ven}\ee
(All the results discussed in the paper have corresponding analogues if instead of (\ref{ven}) one assumes that $u>c$ in $\overline{D_\eta}$.)

 For the remainder of this section we describe informally, following \cite{csold}, two other equivalent formulations of (\ref{Euler}), assuming that the solutions are smooth enough. The equivalence of these formulations under weak regularity assumptions is the main aim of the paper, which will be addressed in the subsequent sections.

\subsection{The stream function formulation}

Suppose that (\ref{Euler}) and (\ref{ven}) hold. Equation (\ref{eq:masscons}) implies the existence of a function $\psi$ in $\overline{D_\eta}$, called a (relative) stream function, such that
 \begin{equation}
  \psi_y=u-c,\,\,\psi_x=-v \quad\text{in }D_\eta.
  \label{transform:stream}
  \end{equation}
  The boundary conditions (\ref{vb}) and (\ref{eq:velbdry}) imply that $\psi$ is a constant on each of $y=0$ and $y=\eta(x)$. Since $\psi$ is only determined up to an additive constant, one can assume that $\psi=0$ on $y=\eta(x)$, and then we obtain that there exists a constant $p_0$ such that $\psi=-p_0$ on $y=0$.
The condition (\ref{ven}) can be rewritten as
\be \psi_y<0\quad\text{in }\overline{D_\eta},\label{aq}\ee
a consequence of which is that $p_0<0$. After expressing the left-hand side in (\ref{eq:momcons1}) and (\ref{eq:momcons2}) in terms of $\psi$, differentiation of the first of these equations with respect to $y$ and of the second with respect to $x$ allows us to eliminate the pressure, leading to
  \be(\Delta\psi)_x\psi_y=(\Delta\psi)_y\psi_x\quad\text{in }D_\eta,\label{vrt}\ee
  where $\Delta$ denotes the Laplace operator. Note that (\ref{aq}) shows that all the level sets of $\psi$ are graphs over the $x$ coordinate, and (\ref{vrt}) then implies that $\Delta\psi$ is constant on each level set of $\psi$. Thus there exists a function $\gamma:[0,-p_0]\to\Rr$ such that
  \[ -\Delta\psi=\gamma(\psi) \quad\text{in }D_\eta.\]
   (Since the quantity $\omega\stackrel{\rm def}{=}v_x-u_y=-\Delta\psi$ has a physical interpretation as the vorticity of the flow, the function $\gamma$ is  customarily referred to in the literature as the vorticity function.)
   Let\begin{equation}
 \Gamma(p)\stackrel{\rm def}{=}\int_0^p \gamma(-s)\,ds\quad\text{for all }p\in [p_0,0].
 \end{equation}
It is then easy to verify, using (\ref{eq:momcons1})--(\ref{eq:momcons2}), that
\be P+\frac{1}{2}|\nabla\psi|^2+gy-\Gamma(-\psi)=\text{constant}\quad\text{in }D_\eta.\label{bec}\ee
In view of (\ref{pres}) and the fact that $\psi=0$ on $y=\eta(x)$, it follows that
\[|\nabla \psi|^2+2gy=Q \quad\text{on }y=\eta(x),\]
for some constant $Q$. We have therefore obtained the \emph{stream function formulation} of the steady water waves equations, which is to find a domain $D_\eta$ and a function $\psi$ in $D_\eta$ such that
  \begin{subequations}\label{eq:streamform}
  \begin{align}
  \Delta\psi=-\gamma(\psi)\qquad\qquad &\textrm{ in }D_\eta,\\
  \psi=-p_0 \,\quad\qquad\qquad&\textrm{ on }y=0,\\
  \psi=0 \,\,\qquad\qquad\qquad&\textrm{ on } y=\eta(x),\\
  |\nabla \psi|^2+2gy=Q \qquad\qquad\qquad&\textrm{ on }y=\eta(x),
  \end{align}
\end{subequations} for some constants $p_0<0$ and $Q$, and some function $\gamma:[0,-p_0]\to\Rr$.

Conversely, suppose that $\psi$ satisfies (\ref{eq:streamform}) and (\ref{aq}) in a domain $D_\eta$. Then one can define in $D_\eta$ a velocity field $(u,v)$ by (\ref{transform:stream}) and a pressure field $P$ by (\ref{bec}) with a suitable choice of the constant in the right-hand side, and easily check that (\ref{Euler}) and (\ref{ven}) hold.

  \subsection{The height (or Dubreil-Jacotin) formulation} An elegant way to overcome the difficulty that in (\ref{eq:streamform}) the fluid domain $D_\eta$ needs to be found as part of the solution has been first observed by Dubreil-Jacotin: the fact that $\psi$ is constant on the top and the bottom of $D_\eta$ can be used to transform (\ref{eq:streamform}) into a nonlinear elliptic boundary-value problem in a fixed domain. More precisely, suppose that (\ref{eq:streamform}) and (\ref{aq}) hold, and let us consider the partial hodograph (or semi-Lagrangian) mapping
   \be(x,y)\mapsto (q,p)=(x, -\psi(x,y)),\label{map1}\ee
 which is, as a consequence of (\ref{aq}), a bijection between $\overline{D_\eta}$ and the closure of the strip \[R=\{(q,p):q\in\Rr, p\in (p_0,0)\}.\]
 Then the inverse mapping, from $\overline R$ to $\overline{D_\eta}$, necessarily has the form
 \be (q,p)\mapsto (x,y)=(q, h(q,p)),\label{map2}\ee
 for some function $h:\overline R\to \Rr$. More precisely, the following two relations hold:
 \be-\psi(q, h(q,p))=p\text{ for all }(q,p)\in R,\qquad  h(x, -\psi(x,y))=y\text{ for all }(x,y)\in D_\eta.\label{inverse}\ee
 (These relations show that, for each $(q,p)\in R$, one may interpret $h(q,p)$ as the \emph{height} of the streamline $\psi=-p$ above the point $(q,0)$ on the bed.) The condition (\ref{aq}) can be expressed as
 \begin{equation}
 h_p>0\quad\text{in }\overline{R}.
 \label{eq:nondegheight}
 \end{equation}
 Note also that
 \begin{equation}h_q=-\frac{\psi_x}{\psi_y},\,h_p=-\frac{1}{\psi_y},\qquad\,\psi_x=\frac{h_q}{h_p},\,\psi_y=-\frac{1}{h_p},\label{ch1}
 \end{equation}
 and
\begin{equation}\label{ch2}
 \partial_x=\partial_q-\frac{h_q}{h_p}\partial_p,\,\partial_y=\frac{1}{h_p}\partial_p,\qquad\,\partial_q=\partial_x-\frac{\psi_x}{ \psi_y}\partial_y,\,\partial_p=-\frac{1}{\psi_y}\partial_y.
 \end{equation}
 Using these identities, one can easily reformulate (\ref{eq:streamform}) as the following system for the function $h$ defined above:
 \begin{subequations} \label{eq:height}
 \begin{align}
 (1+h_q^2)h_{pp}-2h_qh_p h_{qp}+h^2_ph_{qq}=-\gamma(-p)h^3_p\,\,\qquad\qquad &\textrm{ in }R,\\
  h=0 \,\qquad\qquad\qquad\qquad&\textrm{ on }p=p_0,\\
  1+h_q^2+(2gh-Q)h_p^2=0\qquad\qquad\qquad\qquad &\textrm{ on }p=0.
  \end{align}
  \end{subequations}
  This is the \emph{height (or Dubreil-Jacotin) formulation} of the steady water waves equations.

 Conversely, suppose that $h$ satisfies (\ref{eq:height}) and (\ref{eq:nondegheight}). Let $\eta:\Rr\to\Rr$ be given by
 $\eta(q)=h(q,0)$ for all $q\in\Rr$. Then (\ref{eq:nondegheight}) implies that $(q,p)\mapsto (x,y)=(q, h(q,p))$ is a bijection between $\overline R$ and $\overline{D_\eta}$. Defining $\psi$ by (\ref{inverse}), the formulae (\ref{ch1})--(\ref{ch2}) are valid, and one can easily deduce from (\ref{eq:height}) and (\ref{eq:nondegheight}) that (\ref{eq:streamform}) and (\ref{aq}) hold.

\section{Weak formulations of the steady water waves problem}

\subsection{Weak velocity formulation} For sufficiently smooth functions $\eta$, $u, v,$ and $P$, (\ref{Euler}) is easily seen to be equivalent to
\begin{subequations}\label{wEuler}
  \begin{align}
 (u-c)_x+v_y=0 \,\,\,\,\quad\qquad\qquad\qquad&\text{in }D_\eta,\label{weq:masscons}\\
 ((u-c)^2)_x+((u-c)v)_y=-P_x\,\qquad\qquad\qquad&\text{in }D_\eta,\label{weq:momcons1}\\
  ((u-c)v)_x+(v^2)_y=-P_y-g\,\,\qquad\qquad&\text{in }D_\eta, \label{weq:momcons2}\\
  v=0 \,\,\,\quad\qquad\qquad\qquad&\text{on } y=0,\label{wvb}\\
 v=(u-c)\eta_x  \qquad\qquad&\text{on } y=\eta(x),  \label{weq:velbdry}\\
  P=P_{\textnormal{atm}} \qquad\qquad\qquad&\text{on }y=\eta(x).\label{wpres}
 \end{align}
  \end{subequations}
  However, (\ref{wEuler}) may be given a meaning for functions of weaker regularity than those of (\ref{Euler}), namely by interpreting (\ref{weq:masscons})--(\ref{weq:momcons2}) in the sense of distributions.
  Of particular interest for us will be solutions of  (\ref{wEuler}) with $\eta\in C^{1,\alpha}(\Rr)$ and
  $(u,v,P)\in C^{0,\alpha}(\overline{D_\eta})$ for some $\alpha\in (0,1]$, with (\ref{wvb})--(\ref{wpres}) being satisfied in the classical sense, and (\ref{weq:masscons})--(\ref{weq:momcons2}) being satisfied in the sense of distributions. (Under the same regularity assumptions, it is not clear how to give a meaning directly to (\ref{Euler}), because the multiplication of a distribution by a function of finite differentiability is not well defined.)

 \subsection{Weak stream function formulation}

For sufficiently smooth functions $\psi$ and $\gamma$, the algebraic identity
\[(\psi_x\psi_y)_x-\frac{1}{2}(\psi_x^2-\psi_y^2)_y-(\Gamma(-\psi))_y=\psi_y(\Delta\psi+\gamma(\psi))\]
shows that, in the presence of (\ref{aq}), (\ref{eq:streamform}) is equivalent to
 \begin{subequations}\label{weq:streamform}
  \begin{align}
  (\psi_x\psi_y)_x-\frac{1}{2}(\psi_x^2-\psi_y^2)_y-(\Gamma(-\psi))_y=0\,\,\qquad\qquad\qquad &\textrm{ in }D_\eta,\label{ioq}\\
  \psi=-p_0 \,\quad\qquad\qquad&\textrm{ on }y=0,\label{p1}\\
  \psi=0 \,\,\qquad\qquad\qquad&\textrm{ on } y=\eta(x),\label{p2}\\
  |\nabla \psi|^2+2gy=Q \qquad\qquad\qquad&\textrm{ on }y=\eta(x).\label{p3}
  \end{align}
\end{subequations}
Again, (\ref{ioq}) may be required to hold in the sense of distributions. We will be interested in solutions of (\ref{weq:streamform}) with
$\eta\in C^{1,\alpha}(\Rr)$, $\psi\in C^{1,\alpha}(\overline{D_\eta})$ and $\Gamma\in C^{0,\alpha}([p_0,0])$ for some $\alpha\in (0,1]$, with (\ref{p1})--(\ref{p3}) being satisfied in the classical sense, and (\ref{ioq}) being satisfied in the sense of distributions (with $\psi_x, \psi_y$ being understood in the classical sense).

\subsection{Weak height formulation}
For sufficiently smooth functions $h$ and $\gamma$, the algebraic identity
\[\left\{-\frac{1+h_q^2}{2h_p^2}+\Gamma(p)\right\}_p+\left\{\frac{h_q}{h_p}\right\}_q=\frac{1}{h_p^3}
\left\{(1+h_q^2)h_{pp}-2h_qh_p h_{qp}+h^2_ph_{qq}+\gamma(-p)h^3_p\right\}\]
shows that, in the presence of (\ref{eq:nondegheight}), (\ref{eq:height}) is equivalent to
\begin{subequations}
 \label{wehe}
 \begin{align}
  \left\{-\frac{1+h_q^2}{2h_p^2}+\Gamma(p)\right\}_p+\left\{\frac{h_q}{h_p}\right\}_q=0\qquad\qquad\qquad&\text{in }R,\label{w1}\\
h=0\qquad\qquad\qquad&\text{on }p=p_0,\label{w2}\\
  \frac{1+h_q^2}{2h_p^2}+gh-\frac{Q}{2}=0\qquad\qquad\qquad&\text{on }p=0.\label{w3}
  \end{align}
 \end{subequations}
 We will be interested in solutions of (\ref{wehe}) with $h\in C^{1,\alpha}(\overline R)$ and $\Gamma\in C^{0,\alpha}([p_0,0])$ for some $\alpha\in (0,1]$,
 with (\ref{w2})--(\ref{w3}) being satisfied in the classical sense, and (\ref{w1}) being satisfied in the sense of distributions (with $h_p,h_q$ understood
 in the classical sense).

 \section{The equivalence of the weak formulations}

 Weak solutions (in the sense described in the previous section) of the steady water waves problem have been studied only very recently in \cite{cs}.
 That paper deals with waves which are periodic in the horizontal direction, the subscript $per$ being used in what follows to indicate this periodicity requirement. In \cite{cs} the authors develop a global bifurcation theory for weak solutions of (\ref{wehe}) with $h\in C^{1,\alpha}_{\textnormal{per}}(\overline R)$, under the assumption $\Gamma\in C^{0,\alpha}([p_0,0])$, for some $\alpha\in (0,1)$. These would formally correspond to solutions of the weak velocity formulation with $\eta\in C^{1,\alpha}_{\textnormal{per}}(\Rr)$ and $u,v,P\in C^{0,\alpha}_{\textnormal{per}}(\overline{D_\eta})$. However, no rigorous proof of this equivalence is given in \cite{cs}. The only result there on the equivalence of the weak formulations, see
 \cite[Theorem 2]{cs}, is the following:

   {\it Let $0<\alpha<1$ and $r=\frac{2}{1-\alpha}$. Then the following are equivalent:
\begin{itemize}
\item[(i)] the weak velocity formulation (\ref{wEuler}) together with (\ref{ven}), for $\eta\in C^{1,\alpha}_{\textnormal{per}}(\mathbb{R})$ and $u,v,P\in W^{1,r}_{\textnormal{per}}(D_\eta)\subset C^{0,\alpha}_{\textnormal{per}}(\overline{D_\eta})$;
 \item[(ii)] the stream function formulation (\ref{eq:streamform}) together with (\ref{aq}), for $\gamma\in L^r[0,-p_0]$, $\eta\in C^{1,\alpha}_{\textnormal{per}}(\mathbb{R})$ and $\psi\in W^{2,r}_{\textnormal{per}}(D_\eta)\subset C^{1,\alpha}_{\textnormal{per}}(\overline{D_\eta})$;
 \item[(iii)] the weak height formulation (\ref{wehe}) together with (\ref{eq:nondegheight}), for $\Gamma\in W^{1,r}[p_0,0]$ and $h\in W^{2,r}_{\textnormal{per}}(R)\subset C^{1,\alpha}_{\textnormal{per}}(\overline R)$.
\end{itemize}}
 As one can see, in the above result the velocity field $(u,v)$, the pressure $P$, the stream function $\psi$, the height $h$, and the function $\Gamma$, are assumed to have more regularity, namely an additional weak (Sobolev space) derivative, than one would really like.

Our main result, given below, proves the equivalence of the weak formulations under the `right' regularity assumptions, albeit only for the case when the H\"{o}lder exponent satisfies $\alpha\in (1/3, 1]$. (In particular, under our assumptions, the function $\Gamma$ need not have a (weak) derivative.) While the weak stream function and the weak height formulations will be seen to be, in fact, equivalent for any $\alpha\in (0,1]$, it remains an open problem whether the weak velocity and the weak stream function formulations are equivalent for $\alpha\in (0,1/3]$ also.
For simplicity, we state our result for solutions which are periodic in the horizontal direction, though this assumption is not essential, and the result can be easily extended to cover other important situations, for example that of solitary waves \cite{Hur}.

\begin{theorem}\label{tmain}
 Let $\alpha\in (1/3, 1]$. Then the following are equivalent:
\begin{itemize}
\item[(i)] the weak velocity formulation (\ref{wEuler}) together with (\ref{ven}), for $\eta\in C^{1,\alpha}_{\textnormal{per}}(\mathbb{R})$ and $u,v,P\in C^{0,\alpha}_{\textnormal{per}}(\overline{D_\eta})$;
 \item[(ii)] the weak stream function formulation (\ref{weq:streamform}) together with (\ref{aq}), for $\Gamma\in C^{0,\alpha}([p_0,0])$, $\eta\in C^{1,\alpha}_{\textnormal{per}}(\mathbb{R})$ and $\psi\in C^{1,\alpha}_{\textnormal{per}}(\overline{D_\eta})$;
 \item[(iii)] the weak height formulation (\ref{wehe}) together with (\ref{eq:nondegheight}), for $\Gamma\in C^{0,\alpha}([p_0,0])$ and $h\in C^{1,\alpha}_{\textnormal{per}}(\overline R)$.
\end{itemize}
\end{theorem}

A key ingredient in our proof is regularization of the relevant equations.
Therefore, we start with some background results on regularization, valid in any number $d$ of dimensions, where we use $x$, $y$ and $z$ to denote points in $\Rr^d$. Let $\varrho\in C_0^\infty(\mathbb{R}^d)$ be a given function, such that\[\text{  $\varrho\geq 0$ in $\Rr^d$, $\textrm{supp}\,\varrho\subset B_1(0)$, $\varrho(x)=\varrho(-x)$ for all $x\in \Rr^d$, and $\int_{\Rr^d}\varrho(x)\,dx=1$} \] and let us denote $\varrho^\varepsilon\stackrel{\rm def}{=}\frac{1}{\varepsilon^d}\varrho(\frac{x}{\varepsilon})$. Let $V$ be an open set in $\Rr^d$, and consider, for any $\eps>0$, the set
$V^\eps\stackrel{\rm def}{=}\{x\in V: \textnormal{dist}(x, \Rr^d\setminus V)>\eps\}$. For any $f\in L^1_{loc}(V)$ and any $\eps>0$ such that $V^\eps$ is not empty, consider in $V^\eps$ the function
 \begin{align} f^\varepsilon(x)&\stackrel{\rm def}{=}f* \varrho^\varepsilon(x)\label{not}\\&=\int_V \varrho^\varepsilon(x-y)f(y)\,dy\nonumber\\&=\int_{B_1(0)}\varrho(z)f(x-\eps z)\,dz\qquad\text{for all }x\in V^\eps.\nonumber\end{align}
 For any $f,g\in L^2_{loc}(V)$, we also introduce in $V^\eps$ the function
\begin{equation} r^\eps(f,g)(x)\stackrel{\rm def}{=}\int_{B_1(0)}\varrho(z)(f(x-\eps z)-f(x))(g(x-\eps z)-g(x))\,dz\quad\text{for all }x\in V^\eps.\label{eq:r}
\end{equation}
We further denote
\begin{equation}
R^\eps(f,g)\stackrel{\rm def}{=}r^\eps (f,g)-(f-f^\varepsilon)(g-g^\varepsilon)\quad\text{in }V^\eps.
\label{eq:R}
\end{equation}
Then one can easily check that we have, at every point in $V^\eps$,
 \begin{equation}
 (fg)^\varepsilon=f^\varepsilon g^\varepsilon+R^\eps (f,g).\label{care}
 \end{equation}
 (It may be worth pointing out that, while it is not immediately clear from (\ref{eq:R}) that $R^\eps(f,g)$ is a smooth function in $V^\eps$, this smoothness becomes obvious from (\ref{care}).)

\begin{lemma}
Let $V$ be an open set in $\Rr^d$, and $f,g\in C^{0,\alpha}_{\textnormal{loc}}(V)$. Let $K$ be a compact subset of $V$, and let $\eps_0\stackrel{\rm def}{=}\textnormal{dist}(K,\Rr^d\setminus V)/2$ and $K_{0}\stackrel{\rm def}{=}\{x\in \Rr^d: \textnormal{dist}(x,K)\leq \eps_0\}$.
Then there exists a constant $C$ such that, in the notation (\ref{not}), the following estimates hold for all $\eps\in (0,\eps_0)$:
 \begin{itemize}
 \item[(i)] $$\|f^\varepsilon-f\|_{C^0(K)}\le C\varepsilon^\alpha \|f\|_{C^{0,\alpha}(K_{0})},$$
 \item[(ii)] $$\|\nabla f^\varepsilon\|_{C^0(K)}\le C\varepsilon^{\alpha-1}\|f\|_{C^{0,\alpha}(K_{0})},$$
 \item[(iii)] $$\|R^\eps(f,g)\|_{C^0(K)}\le C \eps^{2\alpha}\|f\|_{C^{0,\alpha}(\mathbb{R}^d)}\|g\|_{C^{0,\alpha}(K_{0})}.$$
 \end{itemize}
 \label{lemma:approx}
\end{lemma}

\begin{proof}[Proof of Lemma \ref{lemma:approx}] Note first that, for any $\eps\in (0,\eps_0)$, $K$ is a subset of $V^\eps$, so that $f^\eps$, $g^\eps$ and $R^\eps(f,g)$ are well defined and smooth on an open set containing $K$.

(i) For every $x\in K$, we have
\begin{align*}
|f^\varepsilon(x)-f(x)|&=|\int_{B_1(0)} \varrho(z)(f(x-\varepsilon z)-f(x))\,dz|\\
&\le \|f\|_{C^{0,\alpha}(K_0)} \int_{B_1(0)} \varrho(z) \varepsilon^\alpha |z|^\alpha\,dz\\
&\le C\varepsilon^\alpha \|f\|_{C^{0,\alpha}(K_0)}
\end{align*}

 (ii) For every $x\in K$, we have, using the fact that $\varrho$ is compactly supported in $B_1(0)$, that
\begin{align*}
|\nabla f^\varepsilon (x)|&=|\frac{1}{\varepsilon^{d+1}}\int_{V} \nabla\varrho(\frac{x-y}{\varepsilon})f(y)\,dy|\\
&=\frac{1}{\varepsilon}|\int_{B_1(0)} \nabla\varrho(z)f(x-\varepsilon z)\,dz|\\&=\frac{1}{\varepsilon}|\int_{B_1(0)}\nabla\varrho(z)\left(f(x-\varepsilon z)-f(x)\right)\,dz|\\
&\le\frac{1}{\varepsilon} \|f\|_{C^{0,\alpha}(K_0)} \int_{\mathbb{R}^d} |\nabla \varrho (z)| \varepsilon^\alpha |z|^\alpha\,dz\\&\le C\varepsilon^{\alpha-1} \|f\|_{C^{0,\alpha}(K_0)}
\end{align*}

 (iii) Using (\ref{eq:r}) and (\ref{eq:R}) and reasoning analogously as in the proof of (i) we obtain the required estimate.
 \end{proof}

 \begin{lemma} Let $V$ be an open set in $\Rr^d$ and $f_1,...,f_d\in L^1_{\textnormal{loc}}(V)$. We use the notation $\partial_k$, $k\in\{1,...,d\}$, to denote the partial derivative with respect to the $k$th variable, either in the classical sense, or in the sense of distributions. Suppose that
 \be\sum_{k=1}^d\partial_k f_k=0 \quad\text{in the sense of distributions in $V$}.\label{divw}\ee
 Then, for any $\eps>0$ such that $V^\eps$ is non-empty, where $V^\eps\stackrel{\rm def}{=}\{x\in V:\textnormal{dist}(x, \Rr^d\setminus V)>\eps\}$,
 we have:
 \be\sum_{k=1}^d\partial_k f_k^\eps=0 \quad\text{in the classical sense in $V^\eps$}.\label{divc}\ee
 \label{lemma:reg}
 \end{lemma}

 \begin{proof}
 Fix any $\eps>0$ such that $V^\eps$ is not empty. Let $\varphi\in C^\infty_0(V^\eps)$ be arbitrary, and observe that $\varphi^\eps\in C_0^\infty(V)$,
 where $\vfi^\eps\stackrel{\rm def}=\vfi*\varrho^\eps$ in $\Rr^d$. Assumption (\ref{divw}) implies that
 $$\sum_{k=1}^d \int_{V} f_k(x) \partial_k \varphi^\eps(x)\,dx=0.$$  Using the fact that the regularization operator commutes with differentiation on smooth functions (see for instance \cite[Chapter 3]{MB}), then Fubini's Theorem, and then integration by parts, we obtain
\begin{align*}
0&=\sum_{k=1}^d \int_{V} f_k(x) (\partial_k \varphi)^\varepsilon(x)\,dx\\&=\sum_{k=1}^d \int_{V} f_k(x)\left(\int_{V^\eps}\partial_k\vfi(y)\varrho^\eps(x-y)\,dy\right)\,dx\\
&=\sum_{k=1}^d \int_{V^\eps} \partial_k\vfi(y)\left(\int_{V}f_k(x)\varrho^\eps(y-x)\,dx\right)\,dy\\
&=\sum_{k=1}^d \int_{V^\eps}f_k^\eps(y)\partial_k\vfi(y)\,dy\\
&=- \int_{V^\eps}\left(\sum_{k=1}^d\partial_k f_k^\varepsilon\
(y)\right)\varphi(y)\,dy.
\end{align*}
Since $\varphi\in C^\infty_0(V^\eps)$ was arbitrary, the required conclusion (\ref{divc}) follows.

 \end{proof}

After these preliminaries on regularization, we are now in a position to give the proof of our main result.

\begin{proof}[Proof of Theorem \ref{tmain}]

We prove first the equivalence of (ii) and (iii) which, as we shall see, is valid for any $\alpha\in (0,1]$.

Suppose that (ii) holds. Let $\psi\in C^{1,\alpha}_{\textnormal{per}}(\overline{D_\eta})$ be such that (\ref{weq:streamform}) and (\ref{aq}) hold, where $\Gamma\in C^{0,\alpha}([p_0,0])$.
Defining $h$ as in Section 2, we then have that $h\in C^{1,\alpha}_{\textnormal{per}}(\overline R)$, and the formulae (\ref{ch1})--(\ref{ch2}) are still valid.
(We regard (\ref{ch2}) as a relation between the classical derivatives of any $C^1$ function with respect to the $(x,y)$ variables and those with respect to the
$(q,p)$ variables, and do not assign to it any meaning in the sense of distributions.) Clearly (\ref{p1})--(\ref{p3}) imply (\ref{w2})--(\ref{w3}), and (\ref{aq}) implies (\ref{eq:nondegheight}). Let us now write explicitly the weak form of (\ref{w1}), which we need to prove: for any $\tvfi\in C_0^1(R)$,
\begin{equation}\int_R \left(-\frac{1+h_q^2}{2 h_p^2}+\Gamma(p)\right)\tvfi_p+\frac{h_q}{h_p}\tvfi_q\,dqdp=0.\label{wehet} \end{equation}
For any such $\tvfi$, let $\vfi\in C^1_0(D_\eta)$ be given by $\vfi(x,y)=\tvfi(x,-\psi(x,y))$ for all $(x,y)\in D_\eta$. By changing variables in the integral, using (\ref{ch1})--(\ref{ch2}),
one can rewrite (\ref{wehet}) as
 \be \int_{D_\eta} \Gamma(-\psi) \vfi_y- (\psi_x\psi_y)\vfi_x+\frac{1}{2}(\psi_x^2-\psi_y^2)\vfi_y\,dxdy=0.\label{wdc}\ee
 But (\ref{wdc}) is valid, as a consequence of (\ref{ioq}). This shows that (\ref{w1}) holds. We have thus proved that (iii) holds.

 Suppose now that (iii) holds. Let $h\in C^{1,\alpha}_{\textnormal{per}}(\overline R)$ be such that
 (\ref{wehe}) and (\ref{eq:nondegheight}) hold, where $\Gamma\in C^{0,\alpha}([p_0,0])$. Defining $\eta$ and $\psi$ as in Section 2, we then have that $\eta\in C^{1,\alpha}_{\textnormal{per}}(\Rr)$ and $\psi\in C^{1,\alpha}_{\textnormal{per}}(\overline{D_\eta})$, and the formulae (\ref{ch1})--(\ref{ch2}) are still valid.  Clearly (\ref{w2})--(\ref{w3}) imply (\ref{p1})--(\ref{p3}), and  (\ref{eq:nondegheight}) implies (\ref{aq}). The weak form of (\ref{ioq}), which we need to prove,  is written explicitly as (\ref{wdc}), for any $\vfi\in C^1_0(D_\eta)$. For any such $\vfi$, let $\tvfi\in C^1_0(R)$ be given by $\tvfi(q,p)=\vfi(q, h(q,p))$ for all $(q,p)\in R$. By changing variables in the integral, using (\ref{ch1})--(\ref{ch2}),
one can rewrite (\ref{wdc}) as (\ref{wehet}). But (\ref{wehet}) is valid, as a consequence of (\ref{w1}). This shows that (\ref{ioq}) holds. We have thus proved that (ii) holds.

We now prove the equivalence of (i) and (ii), making essential use of the assumption $\alpha>1/3$.

Suppose that (i) holds. Since $\eta\in C^{1,\alpha}_{\textnormal{per}}(\Rr)$ and $u,v\in C^{0,\alpha}_{\textnormal{per}}(\overline D_\eta)$, it follows from (\ref{weq:masscons}), by arguments similar to those in \cite[Lemma 3]{BBM}, in which our Lemma 2 plays a key role, that there exists $\psi\in C^{1,\alpha}_{\textnormal{per}}(\overline{D_\eta})$, uniquely determined up to an additive constant, such that (\ref{transform:stream}) holds. Clearly, (\ref{ven}) implies (\ref{aq}).
 Also, it follows from (\ref{wvb}) and (\ref{weq:velbdry}) that $\psi$ is constant on each of $y=0$ and $y=\eta(x)$. The additive constant in the definition of $\psi$ may be chosen so that (\ref{p2}) holds, and then (\ref{p1}) also holds for some constant $p_0<0$.
Using the definition of $\psi$ we rewrite (\ref{weq:momcons1})--(\ref{weq:momcons2}) in the weak distributional form (with $\psi_x, \psi_y$ in the classical sense):
\begin{subequations}\label{eq:weakstream}
 \begin{align}
 (\psi_y^2)_x-(\psi_x\psi_y)_y&=-P_x\qquad\qquad\qquad\text{in }D_\eta,\\
 -(\psi_x\psi_y)_x+(\psi_x^2)_y&=-P_y-g\qquad\qquad\,\,\text{in }D_\eta.
 \end{align}
 \end{subequations}
 Let us denote
\begin{equation}
F\stackrel{\rm def}{=}P+\frac{1}{2}|\nabla \psi|^2+gy\quad\text{in }D_\eta.
\label{eq:f}
\end{equation}
It follows from
(\ref{eq:weakstream}) that we have, in the sense of distributions (with $\psi_x, \psi_y$ in the classical sense):
\begin{subequations} \label{eqf}
\begin{align}
F_x&=\frac{1}{2}(\psi_x^2-\psi_y^2)_x+(\psi_x\psi_y)_y\quad\text{in }D_\eta,\\
F_y&=(\psi_x\psi_y)_x-\frac{1}{2}(\psi_x^2-\psi_y^2)_y\quad\text{in }D_\eta.
 \end{align}
 \end{subequations}
We now show that there exists a function $\Gamma\in C^{0,\alpha}([p_0,0])$ such that
  \be F(x,y)=\Gamma(-\psi(x,y))\quad\text{for all }(x,y)\in D_\eta.\label{raa}\ee
Let us consider again the transformations (\ref{map1})--(\ref{inverse}), and note that (\ref{ch1})--(\ref{ch2}) are valid under the present regularity assumptions.
Let $\tilde F:\overline R\to \Rr$ be given by $\tilde F(q,p)=F(q, h(q,p))$ in $\overline R$, which is equivalent to
$F(x,y)=\tilde F(x,-\psi(x,y))$ in $\overline{D_\eta}$. Then our desired conclusion (\ref{raa}) is that
\be \tilde F(q,p)=\Gamma(p)\quad\text{for all }(q,p)\in R,\label{qaa}\ee
for some $\Gamma\in C^{0,\alpha}([p_0,0])$. To this aim, we shall prove that, for any $\tvfi\in C^1_0(R)$,
  \be\int_R \tilde F\tilde\varphi_q\,dqdp=0,\label{wsa}\ee
  which, together with the condition $\tilde F\in C^{0,\alpha}_{\textnormal{per}}(\overline R)$, will imply (\ref{qaa}) for some function $\Gamma\in  C^{0,\alpha}([p_0,0])$.
 For any such $\tvfi$, let $\vfi\in C^1_0(D_\eta)$ be given by $\vfi(x,y)=\tvfi(x,-\psi(x,y))$ for all $(x,y)\in D_\eta$. By changing variables in the integral, using (\ref{ch1}) and (\ref{ch2}), (\ref{wsa}) can be rewritten as
  \begin{equation}
\int_{D_\eta} F\left(\psi_y\vfi_x -\psi_x\varphi_y\right)\,dxdy=0
\label{eq:conclusion}
\end{equation}  Thus our aim is to prove (\ref{eq:conclusion}) for any $\vfi\in C^1_0(D_\eta)$. Note for later reference that this statement can be written in the sense of distributions (with $\psi_x,\psi_y$ in the classical sense) as
\be (F\psi_y)_x-(F\psi_x)_y=0\quad\text{in }D_\eta.\label{nei}\ee
Note also that, for any function $\theta$ of class $C^2$, a direct calculation shows the identities
\begin{subequations}\label{idtheta}
 \begin{align}\frac{1}{2}(\theta_x^2-\theta_y^2)_x+(\theta_x\theta_y)_y&=\theta_x\Delta \theta,\label{idtheta1} \\(\theta_x\theta_y)_x-
 \frac{1}{2}(\theta_x^2-\theta_y^2)_y&=\theta_y\Delta\theta.\label{idtheta2}
 \end{align}
 \end{subequations}
  (The expressions in the left-hand side of (\ref{idtheta}) are similar to those occurring in the right-hand side of (\ref{eqf}), however (\ref{idtheta}) cannot be applied with $\theta:=\psi$, since $\psi$ is not of class $C^2$ in $D_\eta$.)
  Let $V\stackrel{\rm def}{=}D_\eta$ and let, for any $\eps>0$, let  $V^\eps\stackrel{\rm def}{=}\{(x,y)\in V: \textnormal{dist}((x,y), \Rr^2\setminus V)>\eps\}$.
  Using Lemma~\ref{lemma:reg} the system (\ref{eqf}) implies that, for any $\eps>0$ such that $V^\eps$ is not empty, in the notation (\ref{not}),
 \begin{align*} F_x^\eps&=\frac{1}{2}((\psi_x^\eps)^2-(\psi_y^\eps)^2)_x+(\psi_x^\eps\psi_y^\eps)_y+\frac{1}{2}R^\eps(\psi_x,\psi_x)_x
 -\frac{1}{2}R^\eps(\psi_y,\psi_y)_x+R^\eps(\psi_x,\psi_y)_y\quad\text{in }V^\eps,
 \\
F_y^\eps&=(\psi_x^\eps\psi_y^\eps)_x-\frac{1}{2}((\psi_x^\eps)^2-(\psi_y^\eps)^2)_y+R^\eps(\psi_x,\psi_y)_x-\frac{1}{2}R^\eps(\psi_x,\psi_x)_y
 +\frac{1}{2}R^\eps(\psi_y,\psi_y)_y\quad\text{in }V^\eps.
\end{align*}
 Using the identities (\ref{idtheta}) with $\theta:=\psi^\eps$, the above can be rewritten as
\begin{subequations}\label{nele}
\begin{align} F_x^\eps&=\psi^\eps_x\Delta \psi^\eps+\frac{1}{2}R^\eps(\psi_x,\psi_x)_x
 -\frac{1}{2}R^\eps(\psi_y,\psi_y)_x+R^\eps(\psi_x,\psi_y)_y\quad\text{in }V^\eps,
 \\
F_y^\eps&=\psi^\eps_y\Delta\psi^\eps+R^\eps(\psi_x,\psi_y)_x-\frac{1}{2}R^\eps(\psi_x,\psi_x)_y
 +\frac{1}{2}R^\eps(\psi_y,\psi_y)_y\quad\text{in }V^\eps.
\end{align}
\end{subequations}
Let $\vfi\in C^1_0(D_\eta)$, arbitrary, and let $K\stackrel{\rm def}{=}\textnormal{supp}\, \vfi$. Let $\eps_0\stackrel{\rm def}{=}\textnormal{dist}(K,\Rr^2\setminus V)/2$ and $K_{0}\stackrel{\rm def}{=}\{(x,y)\in \Rr^2: \textnormal{dist}((x,y),K)\leq \eps_0\}$. Note that $K$ is a subset of $V^\eps$, for any $\eps\in (0,\eps_0)$. Aiming to prove (\ref{eq:conclusion}) for $\vfi$, we write, for any $\eps\in (0,\eps_0)$,
\begin{align*}&\int_{D_\eta} F\left(\psi_y\vfi_x -\psi_x\varphi_y\right)\,dxdy\\&=\int_{K} (F\psi_y-F^\eps\psi_y^\eps)\vfi_x -(F\psi_x-F^\eps\psi_x^\eps)\varphi_y\,dxdy+\int_{K} F^\eps\psi_y^\eps\vfi_x -F^\eps\psi_x^\eps\varphi_y\,dxdy\\&\stackrel{\rm def}{=}\mathcal{I}_\eps+\mathcal{J}_\eps.
\end{align*}
It is a consequence of Lemma 1(i) that $\mathcal{I}_\eps\to 0$ as $\eps\to 0$.
 To estimate $\mathcal{J}_\varepsilon$, we first integrate by parts, then use (\ref{nele}) to cancel some terms, and then integrate by parts again, to get
\begin{align*}
\mathcal{J}_\varepsilon&=\int_{K}(F^\varepsilon_x\psi^\varepsilon_y-F^\varepsilon_y\psi^\varepsilon_x)\varphi\,dxdy\\
&=-\int_{K}\left[\frac{1}{2}R^\eps(\psi_x,\psi_x)_x
 \frac{1}{2}R^\eps(\psi_y,\psi_y)_x+R^\eps(\psi_x,\psi_y)_y\right](\psi^\eps_y\vfi)\,dxdy\\
 &\quad+\int_{K} \left[R^\eps(\psi_x,\psi_y)_x-\frac{1}{2}R^\eps(\psi_x,\psi_x)_y
 +\frac{1}{2}R^\eps(\psi_y,\psi_y)_y\right](\psi^\eps_x\vfi)\,dxdy\\
 &=\int_{K} \left[\frac{1}{2}R^\eps(\psi_x,\psi_x)-\frac{1}{2}R^\eps(\psi_y,\psi_y)\right][(\psi^\eps_y\vfi)_x+(\psi^\eps_x\vfi)_y]\,dxdy\\&\quad+\int_{K} R^\eps(\psi_x,\psi_y)[(\psi^\eps_y\vfi)_y-(\psi^\eps_x\vfi)_x]\,dxdy.
\end{align*}
Expanding the square brackets, we write $\mathcal{J}_\eps$ as a sum of six terms, all of which can be estimated in a similar way, by using Lemma 1, to the one shown below:
\begin{align*}
|\int_{K} R^\varepsilon(\psi_x,\psi_y)(\psi_{x}^\varepsilon\vfi)_x\,dx|=|\int_{K} R^\varepsilon(\psi_x,\psi_y)\big(\psi_{x}^\varepsilon\vfi_x+ \psi_{xy}^\eps\vfi\big)\,dx|\\
\le C(\eps^{2\alpha} \|\psi_x\|^2_{C^{0,\alpha}(K_0)}\|\psi_y\|_{C^{0,\alpha}(K_0)}+\eps^{3\alpha-1}\|\psi_x\|_{C^{0,\alpha}(K_0)}^3),
 \end{align*} where $C$ is a constant which depends on $\vfi$, but is independent of $\eps\in (0,\eps_0)$.
 The assumption $\alpha>1/3$ now implies that $\mathcal{J}_\eps\to 0$ as $\eps\to 0$. We have thus proved that (\ref{eq:conclusion}) holds for any $\vfi\in C^1_0(D_\eta)$. As discussed earlier, this implies the existence of $\Gamma\in C^{0,\alpha}([p_0,0])$ such that (\ref{raa}) holds. It therefore follows from (\ref{eqf}) that, in the sense of distributions (with $\psi_x,\psi_y$ in the classical sense),
\begin{subequations} \label{equa}
\begin{align}
\Gamma(-\psi)_x&=\frac{1}{2}(\psi_x^2-\psi_y^2)_x+(\psi_x\psi_y)_y\quad\text{in } D_\eta,\label{equa1}\\
\Gamma(-\psi)_y&=(\psi_x\psi_y)_x-\frac{1}{2}(\psi_x^2-\psi_y^2)_y \quad\text{in } D_\eta.\label{equa2}
 \end{align}
 \end{subequations}
 Rearranging (\ref{equa2}) gives exactly (\ref{ioq}). Also, recalling (\ref{eq:f}), we obtain from (\ref{wpres}) the validity of (\ref{p3}) for some constant $Q$. We have thus proved that (ii) holds.

Suppose now that (ii) holds. We define in $D_\eta$ the velocity
$(u,v)$ by (\ref{transform:stream}) and, up to an additive
constant, the pressure $P$, by \be P\stackrel{\rm
def}{=}-\frac{1}{2}|\nabla\psi|^2-gy+\Gamma(-\psi)\quad\text{in }D_\eta.\label{fre}\ee
Then $u,v, P\in C^{0,\alpha}_{\textnormal{per}}(\overline{D_\eta})$. Moreover,
the definition of $u$ and $v$ implies (\ref{weq:masscons}), while
(\ref{p1})--(\ref{p3}) imply (\ref{wvb})--(\ref{wpres}), provided the additive constant in the definition of $P$ is chosen in a suitable way.
 Also,
(\ref{aq}) implies (\ref{ven}). It therefore remains to prove the
validity of (\ref{weq:momcons1})--(\ref{weq:momcons2}). Using our
definition of $u,v$ and $P$, (\ref{weq:momcons1})--(\ref{weq:momcons2}) can be equivalently rewritten
as (\ref{equa}). However, (\ref{equa2}) is exactly (\ref{ioq}),
which we are assuming to hold, and therefore it only remains to
prove (\ref{equa1}). We now show that (\ref{equa2}) implies
(\ref{equa1}). For notational convenience, we denote
$F\stackrel{\rm def}{=}\Gamma(-\psi)$. We claim that, with this
definition of $F$, (\ref{nei}) necessarily holds. Indeed,
(\ref{nei}) can be written explicitly as (\ref{eq:conclusion}) for any
$\vfi\in C^1_0(D_\eta)$, which, using the same notation as earlier
in the proof, is equivalent to (\ref{wsa}) for any $\tvfi\in
C^1_0(R)$, which is clearly true with our definition of $F$. Let $V\stackrel{\rm def}{=}D_\eta$ and let, for any $\eps>0$, let  $V^\eps\stackrel{\rm def}{=}\{(x,y)\in V: \textnormal{dist}((x,y), \Rr^2\setminus V)>\eps\}$.
Using Lemma~\ref{lemma:reg}, (\ref{nei}) implies that, for any $\eps>0$ such that $V^\eps$ is not empty, in the notation (\ref{not}),
\be (F^\eps\psi_y^\eps)_x-(F^\eps\psi_x^\eps)_y+ R^\eps(F,\psi_y)_x-R^\eps(F,\psi_x)_y=0\quad\text{in }V^\eps.\label{and}\ee
Let $K$ be any compact subset of $V$. It follows from (\ref{aq}) that, for all $\eps$ sufficiently small,
\[\psi^\eps_y<0\quad\text{in $K$}.\]
For such values of $\eps$, (\ref{and}) can be rewritten as
\be F^\eps_x=\frac{\psi^\eps_x}{\psi^\eps_y}F^\eps_y-\frac{1}{\psi^\eps_y}[R^\eps(F,\psi_y)_x-R^\eps(F,\psi_x)_y]\quad\text{in }K.\label{iu}\ee
 Using again Lemma~\ref{lemma:reg} and (\ref{equa2}), we obtain, by using (\ref{idtheta2}) with $\theta:=\psi^\eps$, that, for any $\eps>0$ such that $V^\eps$ is not empty,
\be F^\eps_y=\psi^\eps_y\Delta\psi^\eps+R^\eps(\psi_x,\psi_y)_x-\frac{1}{2}R^\eps(\psi_x,\psi_x)_y
 +\frac{1}{2}R^\eps(\psi_y,\psi_y)_y\label{lo}\quad\text{in }V^\eps.\ee
 We deduce from (\ref{lo}) using (\ref{iu}) that, for all $\eps$ sufficiently small,
 \begin{align} F^\eps_x=\psi^\eps_x \Delta \psi^\eps&+\frac{\psi^\eps_x}{\psi^\eps_y}\left[R^\eps(\psi_x,\psi_y)_x-\frac{1}{2}R^\eps(\psi_x,\psi_x)_y
 +\frac{1}{2}R^\eps(\psi_y,\psi_y)_y\right]\nonumber\\&-\frac{1}{\psi^\eps_y}[R^\eps(F,\psi_y)_x-R^\eps(F,\psi_x)_y]\quad\text{in }K.\label{lala}\end{align}
We now write explicitly the weak form of (\ref{equa1}), which we want to prove: for any $\vfi\in C_0^1 (D_\eta)$,
\be\int_{D_\eta} F\vfi_x-\frac{1}{2}(\psi_x^2-\psi_y^2)\vfi_x-(\psi_x\psi_y)\vfi_y\,dxdy=0.\label{tgb}\ee
Let $\vfi\in C_0^1 (D_\eta)$, arbitrary, and let $K\stackrel{\rm def}=\textnormal{supp}\,\vfi$. Let $\eps_0\stackrel{\rm def}{=}\textnormal{dist}(K,\Rr^2\setminus V)/2$ and $K_{0}\stackrel{\rm def}{=}\{(x,y)\in \Rr^2: \textnormal{dist}((x,y),K)\leq \eps_0\}$. Note that $K$ is a subset of $V^\eps$, for any $\eps\in (0,\eps_0)$. Aiming to prove (\ref{tgb}) for $\vfi$, we write, for any $\eps\in (0,\eps_0)$,
\begin{align*}
&\int_{D_\eta} F\vfi_x-\frac{1}{2}(\psi_x^2-\psi_y^2)\vfi_x-(\psi_x\psi_y)\vfi_y\,dxdy\\&=
\int_{K} [F-F^\eps]\vfi_x-\left[\frac{1}{2}(\psi_x^2-\psi_y^2)-\frac{1}{2}((\psi^\eps_x)^2-(\psi^\eps_y)^2)\right]\vfi_x-[(\psi_x\psi_y)-(\psi_x^\eps\psi_y^\eps)]
\vfi_y\,dxdy\\&\quad+\int_{K} F^\eps\vfi_x-\frac{1}{2}((\psi^\eps_x)^2-(\psi^\eps_y)^2)\vfi_x-(\psi_x^\eps\psi_y^\eps)\vfi_y\,dxdy\\
&\stackrel{\rm def}{=}\mathcal{K}_\eps+\mathcal{L}_\eps.
\end{align*}
It is a consequence of Lemma 1(i) that $\mathcal{K}_\eps\to 0$ as $\eps\to 0$. Now note that (\ref{aq}) implies that there exists $\tilde\eps\in (0,\eps_0)$ and $\delta>0$ such that, for all $\eps\in (0,\tilde\eps)$,
\be\psi_y^\eps\geq\delta \quad\text{in $K$}.\label{sra}\ee
  To estimate $\mathcal{L}_\eps$,  we first integrate by parts using (\ref{idtheta1}) with $\theta:=\psi^\eps$, then use (\ref{lala}) to cancel some terms, and then integrate by parts again, to obtain, for any $\eps\in (0,\tilde\eps)$,
\begin{align*} \mathcal{L}_\eps&=-\int_{K} (F^\eps_x-\psi^\eps_x\Delta\psi^\eps)\vfi\\&=-\int_{K}\frac{\psi^\eps_x}{\psi^\eps_y}\left[R^\eps(\psi_x,\psi_y)_x-\frac{1}{2}R^\eps(\psi_x,\psi_x)_y
 +\frac{1}{2}R^\eps(\psi_y,\psi_y)_y\right]\vfi\,dxdy\\&\quad+\int_{K}\frac{1}{\psi^\eps_y}\left[R^\eps(F,\psi_y)_x-R^\eps(F,\psi_x)_y\right]\vfi\,dxdy\\
 &=\int_{K} R^\eps(\psi_x,\psi_y)\left(\frac{\psi^\eps_x}{\psi^\eps_y}\vfi\right)_x
 -\frac{1}{2}R^\eps(\psi_x,\psi_x)\left(\frac{\psi^\eps_x}{\psi^\eps_y}\vfi\right)_y- \frac{1}{2}R^\eps(\psi_y,\psi_y)\left(\frac{\psi^\eps_x}{\psi^\eps_y}\vfi\right)_y\,dxdy\\
&\quad -\int_{K}R^\eps(F,\psi_y)\left(\frac{1}{\psi^\eps_y}\vfi\right)_x-R^\eps(F,\psi_x)\left(\frac{1}{\psi^\eps_y}\vfi\right)_y\,dxdy.\end{align*}
Thus we have written $\mathcal{L}_\eps$ as a sum of five terms, all of which can be estimated in a similar way, by using Lemma 1, to the one shown below:
\begin{align*}
|\int_{K} R^\eps(\psi_x,\psi_y)\left(\frac{\psi^\eps_x}{\psi^\eps_y}\vfi\right)_x\,dxdy|=|\int_{K} R^\eps(\psi_x,\psi_y)\left(\frac{\psi^\eps_x\psi^\eps_y}{(\psi^\eps_y)^2}\vfi_x+\frac{\psi^\eps_{xx}\psi^\eps_y
-\psi^\eps_x\psi^\eps_{xy}}{(\psi^\eps_y)^2}\vfi\right)\,dxdy|&\\
\leq C(\eps^{2\alpha}||\psi_x||^2_{C^{0,\alpha}(K_0)}||\psi_y||^2_{C^{0,\alpha}(K_0)}+\eps^{3\alpha-1}||\psi_x||_{C^{0,\alpha}(K_0)}^2||\psi_y||^2_{C^{0,\alpha}(K_0)}),&
\end{align*}
 where $C$ is a constant which depends on $\vfi$, but is independent of $\eps\in (0, \tilde \eps)$, and we have also used (\ref{sra}).
 The assumption $\alpha>1/3$ now implies that $\mathcal{L}_\eps\to 0$ as $\eps\to 0$. We have thus proved that (\ref{tgb}) holds for any $\vfi\in C^1_0(D_\eta)$, and therefore that (\ref{equa1}) holds. This completes the proof that (i) holds.
\end{proof}


\begin{thebibliography}{a}
\bibitem{BBM}  J. Bourgain, H. Brezis, and P. Mironescu, ``Lifting in Sobolev spaces", J. Anal. Math. 80 (2000), 37--86

\bibitem{csold} A. Constantin and W. Strauss, ``Exact steady periodic water waves with vorticity", Comm. Pure Appl. Math. 57 (2004), 481--527
\bibitem{cs} A. Constantin and W. Strauss, ``Periodic traveling gravity water waves with discontinuous vorticity", to appear in Arch. Rational Mech. Anal.
\bibitem{CV} A. Constantin and E. Varvaruca, ``Steady periodic water waves with constant vorticity: regularity and local bifurcation", Arch. Rational Mech. Anal. 199 (2011), 33--67
\bibitem{CET} P. Constantin, W. E, and E. Titi, ``Onsager's conjecture on the energy conservation for
solutions of Euler's equation", Commun. Math. Phys. 165 (1994), 207--209
\bibitem{EEW} M. Ehrnstr\"{o}m, J. Escher, and E. Wahl\'{e}n, ``Steady water waves with multiple critical layers", preprint 2010
\bibitem{Hur} V. M. Hur, ``Exact solitary water waves with vorticity", Arch. Rational Mech. Anal. 188 (2008), 213--244
\bibitem{MB} A. Majda and A. Bertozzi, ``Vorticity and incompressible flow", Cambridge University Press, Cambridge, 2002
\bibitem{ST} E. Shargorodsky and J. F. Toland, ``Bernoulli free-boundary
problems", Mem. Amer. Math. Soc. 196 (2008), no. 914
\bibitem{S} W. Strauss, ``Steady water waves", Bull. Amer. Math. Soc. 47 (2010), 671--694

\end{thebibliography}
\end{document}